\documentclass{article}
\usepackage{amsmath, amsthm, amsfonts, cancel, mathabx, bbm, graphicx, stmaryrd, tikz-cd, thmtools}
\usepackage[margin=1in]{geometry}

\usepackage[square, numbers, sort&compress]{natbib}
\usepackage[colorlinks, allcolors=blue]{hyperref}
\usepackage[noabbrev, nameinlink]{cleveref}

\newcommand{\ev}[1]{( #1 )} 
\newcommand{\ew}[1]{\left( #1 \right)}

\theoremstyle{definition}
\newtheorem{theorem}{Theorem}[section]
\newtheorem{lemma}[theorem]{Lemma}
\newtheorem{corollary}[theorem]{Corollary}
\newtheorem{remark}{Remark}

\numberwithin{equation}{section}
\newtheorem{definition}{Definition}
\newtheorem{proposition}[theorem]{Proposition}

\title{Multigraphs with Unique Partition into Cycles}
\usepackage{authblk}

\author[1]{Joshua Cooper}
\author[2]{Utku Okur}
\affil[1]{Department of Mathematics, University of South Carolina, U.S.A. (\href{mailto:email}{cooper@math.sc.edu})}
\affil[2]{Department of Mathematics, University of South Carolina, U.S.A. (\href{mailto:email}{uokur@email.sc.edu})}

\date{\today}

\begin{document}
\maketitle

\begin{abstract}
Due to Veblen's Theorem, if a connected multigraph $X$ has even degrees at each vertex, then it is Eulerian and its edge set has a partition into cycles. In this paper, we show that an Eulerian multigraph has a unique partition into cycles if and only if it belongs to the family $\mathcal{S}$, ``bridgeless cactus multigraphs", elements of which are obtained by replacing every edge of a tree with a cycle of length $\geq 2$. Other characterizing conditions for bridgeless cactus multigraphs and digraphs are provided. 

Furthermore, for a digraph $D$, we list conditions equivalent to having a unique Eulerian circuit, thereby generalizing a previous result of Arratia-Bollob\'{a}s-Sorkin. In particular, we show that digraphs with a unique Eulerian circuit constitute a subfamily of $\mathcal{S}$, namely, ``Christmas cactus digraphs".

\textit{Keywords: Digraph, Multigraph, Partition into Cycles} 

\textit{Mathematics Subject Classifications 2020: Primary 05C75; Secondary 05C45, 05C38, 05C20.} 	
\end{abstract}

\hypersetup{linkcolor=black} 

\tableofcontents

\hypersetup{linkcolor=blue}

\section{Introduction}
Consider a multigraph $X$ without loops. By a theorem of Veblen (\cite[p.~87]{veblen}), if $X$ has even degree at each vertex, then its edge set can be partitioned into edge-disjoint cycles of length $\geq 2$ (A cycle of length $2$ is called \textit{digon}.). The argument goes: Start at any vertex and greedily traverse the edges, as long as the new edge is never visited before. Since there are finitely many edges, the process has to stop. If we never visit a vertex more than once, then the multigraph is actually a path, contradicting the hypothesis that each vertex has even degree. So, we have to visit a vertex for the second time and obtain a cycle. If the multigraph itself is a cycle, then the partition is complete. Otherwise, we subtract the edges of the cycle and as the remainder also satisfies the hypothesis, we can repeat the process to obtain a partition into edge-disjoint cycles. Let us call a multigraph \textit{even} if all vertices of $X$ have even degrees. A natural question arises: Given that a partition into cycles always exists for an even multigraph, then can we say anything about the number of partitions? In this note, we characterize multigraphs $X$ with a unique partition into edge-disjoint cycles. The answer seems intuitively clear: A multigraph has a unique partition into cycles if and only if it is a member of the family $\mathcal{S}$ of multigraphs which is defined as the closure of the set of \textit{all} finite cycles under the operation of appending an edge-disjoint cycle. In other words, a member of $\mathcal{S}$ is obtained by starting with an initial cycle and adding new cycles, at each step, with the condition that the new cycle intersects the old graph in exactly one vertex. This class of graphs is a subclass of the class of ``cactus graphs", which has been considered in the literature (see, for example, \cite{husimi,bahrani}). 

Let us call a multigraph \textit{simple} if it has no parallel edges. In \cite[p.~3]{bahrani}, a \textit{cactus graph} is defined as a simple graph where no two cycles share an edge. Equivalently, if a tree is taken and possibly some of the edges are replaced with a cycle of length $\geq 3$, we obtain a simple cactus graph. Small examples and additional references concerning cactus graphs can be found in \cite{weisstein}. Recall that a \textit{bridge} in a multigraph is an edge such that its removal disconnects the multigraph. Expanding on the terminology of \cite{bahrani}, let us define a \textit{cactus multigraph} to be a multigraph where no two cycles share an edge. If we replace \textit{every} edge of a tree with a cycle of length $\geq 2$, we obtain a bridgeless cactus multigraph. The present manuscript is concerned only with bridgeless cactus undirected multigraphs and directed multigraphs (shortly, \textit{digraphs}). The main theorem (q.v.~\Cref{thm:unique_decomp_equiv_digraph}) states that a connected Eulerian digraph has a unique edge-disjoint partition into cycles if and only if it is a bridgeless cactus digraph.

The original motivation for the characterization of digraphs and multigraphs with a unique partition into cycles is the following: In \cite{co3}, the partition lattice of the edge set $E\ev{D}$ of an Eulerian digraph $D$ is considered. Furthermore, the subposet induced by partitions into Eulerian subdigraphs forms a join-semilattice and is denoted $T\ev{D}$. Partitions of $E\ev{D}$ into disjoint cycles correspond to minimal elements of $T\ev{D}$. It follows from the characterization below and \cite[Lemma~3.1, p.~16]{co3} that $T\ev{D}$ is a lattice if and only if $D$ is a bridgeless cactus digraph. 

In \Cref{sec:unique_eulerian}, we consider connected Eulerian digraphs with a unique Eulerian circuit. It turns out that this condition is equivalent to being a Christmas cactus digraph, which is a special type of a cactus digraph where every vertex belongs to at most two blocks. Applications of Christmas cactus graphs in planar graph embeddings and results on related problems are found in \cite{prutkin, angelini, leighton}. An alternative sufficient condition for having a unique Eulerian circuit is given in \cite{bollobas}, for 2-in, 2-out digraphs (each vertex has in-degree and out-degree $2$), possibly with loops.

For further reference, note that triangular cactus graphs (cactus graphs in which each cycle is a triangle) are investigated in \cite{husimi}, under the name of ``Husimi trees'' and a functional equation for their generating function is obtained. 

For the purposes of exposition, we show our results for digraphs in \Cref{thm:unique_decomp_equiv_digraph} and \Cref{prop:unique_eulerian} of \Cref{sec:digraphs} and derive a counterpart for multigraphs as a corollary (q.v. \Cref{cor:unique_decomp_equiv_multi_graph}) in \Cref{sec:multi_graph_partition}. 

\section{Eulerian Bridgeless Cactus Digraphs}
\label{sec:digraphs}
\subsection{Preliminaries}
In this paper, we write $\mathcal{A}_1 \sqcup \mathcal{A}_2$ for the disjoint union of two sets $\mathcal{A}_1$ and $\mathcal{A}_2$. The power set of a set $\mathcal{A}$ is denoted $\mathcal{P}\ev{\mathcal{A}}$. Define $ \binom{\mathcal{A}}{2} := \{ \{v_1,v_2\} : v_1,v_2 \in \mathcal{A}, \ v_1\neq v_2\}$, for a finite set $\mathcal{A}$. The following is a version of the definition of a graph in \cite[p.~2]{bondy}, stated for loopless multigraphs of rank $2$ (each edge is incident to $2$ vertices).  
\begin{definition}
\label{def:multi_graph}
A \textit{multigraph} $X$ of rank $k = 2$ is a triple $X = \ev{ V\ev{X}, E\ev{X}, \phi }$, where $V\ev{X}, E\ev{X}$ are finite sets and $\phi: E\ev{X} \rightarrow \binom{V\ev{X}}{2} $ is a mapping (there are no loops.)
\begin{enumerate}
\item If $X = \ev{ V\ev{X}, E\ev{X}, \phi }$ is given such that $\phi$ is injective, then $X$ is a \textit{simple} graph. 
\item Two edges $e_1, e_2 \in E\ev{X}$ are \textit{parallel}, provided $\phi\ev{e_1} = \phi\ev{e_2}$. 
\end{enumerate}
\end{definition}

Next, we define multi-directed graphs, or shortly, digraphs: 
\begin{definition}
\label{def:orientations}
A \textit{digraph} $D=\ev{ V\ev{D}, E\ev{D}, \psi}$ of rank $2$ is a triple, where $V\ev{D}, E\ev{D}$ are finite sets and $\psi: E\ev{D} \rightarrow V\ev{D} \times V\ev{D}$ is a function such that coordinates of $\psi\ev{e}$ are distinct, for each edge $e\in E\ev{D}$. (There are no loops.) 
\begin{enumerate}
\item Two edges $e_1, e_2 \in E\ev{D}$ are \textit{parallel}, provided $\psi\ev{e_1} = \psi\ev{e_2}$.
\item The \textit{out-degree} of a vertex $u\in V\ev{D}$ is defined as $\deg_D^{+}\ev{u} = |\{ e\in E\ev{D}: \psi\ev{e} = \ev{u,v} \text{ for some } v\in V\ev{D} \} | $.
\item Given an undirected multigraph $X=(V(X),E(X),\phi)$ and an digraph $O=(V(O),E(O),\psi)$, then $O$ is an \textit{orientation} of $X$, provided $V\ev{O} = V\ev{X}$, $E\ev{O} = E\ev{X}$ and we have $\phi\ev{ e } = \theta\ev{ \psi\ev{ e } }$ for each $e\in E\ev{X} = E\ev{O}$, where 
\begin{align*}
\theta: V\ev{X} \times V\ev{X} &\rightarrow \binom{V\ev{X}}{2} \\ 
\ev{ v_1, v_2 } & \mapsto \{ v_1, v_2\}
\end{align*}
is the ``forgetful'' mapping. The set of orientations of $X$ is denoted as $\mathcal{O}\ev{X}$. So, we have $|\mathcal{O}\ev{X}| = 2^{|E\ev{X}|}$. 
\end{enumerate}
\end{definition}

We provide some preliminary definitions for Eulerian multigraphs. Let $X$ be a multigraph. A \textit{walk} of $X$ is an alternating sequence \[\mathbf{w}=\ev{v_0,e_1,v_1,e_2,\ldots, v_{d-1}, e_{d}, v_d}\] of vertices $\ev{ v_i }_{i=0}^{d}$ and edges $\ev{ e_i }_{i=1}^{d}$, such that consecutive pairs of elements are incident. The number $d$ is the \textit{length} of $\mathbf{w}$. The vertices $v_0,v_d$ are the \textit{initial vertex} and the \textit{terminal vertex} of $\mathbf{w}$, respectively. The walk $\mathbf{w}$ is \textit{closed}, if $v_0 = v_d$ (In this case, $\mathbf{w}$ is a closed walk at $v_0$.). The walk $\mathbf{w}$ is a \textit{trail}, if its edges are pairwise distinct. A closed trail $\mathbf{w}=\ev{v_0,e_1,v_1,e_2,\ldots, v_{d-1}, e_{d}, v_d}$ is \textit{simple}, provided $v_i \neq v_j$ for each $\ev{i,j}$ such that $0\leq i < j \leq d $ and $\ev{i,j} \neq \ev{0,d}$. A closed trail $\mathbf{w}$ of $X$ is \textit{Eulerian}, if it traverses every edge of $X$. We denote by $\mathcal{T}\ev{X}$ and $\mathcal{W}\ev{X}$, the sets of closed trails and Eulerian (closed) trails of $X$, respectively. We define the set of \textit{circuits} of $X$ as the quotient $\mathcal{Z}\ev{X} := \mathcal{T}\ev{X} /\rho$ of closed trails under the equivalence relation $\rho$ of cyclic permutation of edges. More formally, $\rho$ is the transitive closure of the relations:
$$ \mathbf{w}_1 = \ev{ v_0, e_1, v_1, e_2,\ldots, e_{d-1}, v_{d-1}, e_{d}, v_0 } \ \ \rho \ \  \mathbf{w}_2 = \ev{ v_{d-1}, e_{d}, v_0, e_1, v_1,e_2,\ldots, e_{d-1} ,v_{d-1} } $$

The set of \textit{Eulerian circuits} is defined as $\mathfrak{C}\ev{ X } := \mathcal{W}\ev{ X } / \rho$, the quotient of the set of Eulerian trails of $X$, up to the equivalence relation of cyclic permutation of edges. The multigraph $X$ is \textit{Eulerian}, if $\mathfrak{C}\ev{X}\neq \emptyset$. By the definition of an Eulerian circuit, Eulerian multigraphs are necessarily connected. 

Whenever a circuit $[w]_\rho \in \mathcal{Z}\ev{ X } $ has no repeated vertices, it is called a \textit{cycle}. Let $\mathcal{B}\ev{X}$ be the set of cycles of $X$. A cycle of length $d=2$ is a \textit{digon}. If $X$ has two parallel edges $e_1,e_2$, then it follows from the definition of a cycle that there is a unique digon with edge set $\{e_1,e_2\}$. On the other hand, for each set $A:=\{ e_1,\ldots,e_m\}$ of pairwise non-parallel edges with $m\geq 3$, there are exactly $2$ cycles of $X$, of length $m$, with edge set $A$.

A walk of a digraph is defined similarly: A walk $\mathbf{w}=\ev{v_0,e_1,v_1,e_2,\ldots, v_{d-1}, e_{d}, v_d}$ must satisfy $\psi\ev{ e_i } = \ev{ v_{i-1}, v_{i} }$, for each $i = 1,\ldots, d$. We use the same notation, $\mathcal{B}\ev{D}$, for the set of cycles of a digraph $D$. Given a closed trail $\mathbf{w} = \ev{ v_0, e_1,v_1,\ldots, v_{m-1},e_{m-1},v_m} $ and two indices $i \leq j $, then we define the subtrail,
$$
v_i \mathbf{w} v_j = \ev{ v_i, e_i, v_{i+1},\ldots, v_{j-1},e_{j-1},v_j}
$$

\begin{definition}
\label{def:first_simple_closed_trail}
Given a closed trail $\mathbf{w}=\ev{v_0,e_1,v_1,e_2,\ldots, v_{d-1}, e_{d}, v_d}$, then let $j$ be the minimal index such that there exists some $i < j$ with $v_i = v_j$. Then, the subtrail $\mathbf{w}_1 := v_i \mathbf{w} v_j$ is a simple closed trail, called the \textit{first simple closed subtrail} of $\mathbf{w}$. In particular, $\beta := [\mathbf{w}_1]_{\rho}$ is a cycle.
\end{definition}

Given an edge $e\in E\ev{D}$ such that $\psi\ev{e} = \ev{z_1,z_2}$, then we define the set of Eulerian (closed) trails of $D$ ending at $e$:
\begin{align*}
\mathcal{W}^{e}\ev{D} &= \{ \mathbf{w} = \ev{ v_0, e_1, v_1,\ldots, v_{ d-2 }, e_{ d-1 }, z_1, e, z_2 } : \mathbf{w} \in \mathcal{W}\ev{D} \}
\end{align*}

Fixing an edge $e\in D$, note that Eulerian circuits of $D$ are clearly in bijection with the Eulerian trails of $D$, ending at $e$, and so $|\mathfrak{C}\ev{D}| = |\mathcal{W}^{e}\ev{D}|$. 

For a vertex $u\in V\ev{D}$, let $\tau^{u}\ev{D}$ be the set of (spanning) arborescences of $D$, rooted at $u$ (see \cite[p.~34]{schrijver} for the definition of an arborescence). By the B.E.S.T. Theorem (\cite[Theorem 6, p.~213]{best}), the number $|\tau^{u}\ev{D}|$ is constant over the variable $u\in V\ev{D}$, so we may define $\tau_D := \tau^{u}\ev{D}$, for any $u\in V\ev{D}$. Furthermore, the B.E.S.T. Theorem connects the number of Eulerian circuits of a digraph to the number of arborescences:
\begin{theorem}[B.E.S.T. Theorem]
\label{thm:best}
$$
|\mathfrak{C}\ev{D}| = \tau_D \prod_{v\in V\ev{D}} (\deg_D^{+}\ev{v}-1)!
$$
\end{theorem}

\begin{definition}
\label{def:edge_contract}
Let $D = \ev{ V\ev{D}, E\ev{D}, \psi}$ be an Eulerian digraph and $e\in E\ev{D}$ be a fixed edge, with $\psi\ev{e} = \ev{u,v}$. Then, we obtain an Eulerian digraph $D'$, by contracting $e$ and replacing $u$ and $v$ with a new vertex $u'$ (c.f. \cite[p.~194]{gross}, for the definition of contraction). Edges incident to both $u$ and $v$ are deleted, instead of making loops.
\end{definition}
\begin{remark}
\label{rmk:edge_contract_arb_num}
Given an Eulerian digraph $D$ and an edge $e\in E\ev{D}$ with initial vertex $u$, then there is a natural one-to-one correspondence:
$$
\tau^{u}_e\ev{D} \leftrightarrow \tau^{u'}\ev{D'}
$$
where $D'$ is the digraph obtained by contracting $e$, as in \Cref{def:edge_contract} and $\tau^{u}_e\ev{D}$ is the set of arborescences of $D$, rooted at $u$ and containing $e$. 
\end{remark}

\subsection{Eulerian digraphs with a unique partition into cycles}
By Veblen's Theorem (\cite[p.~87]{veblen}), an Eulerian multi-digraph $D$ has at least one partition into cycles. With a series of lemmas, we shall characterize connected Eulerian digraphs $D$ with a unique partition into cycles. 

\begin{remark}
For two digraphs $D_1,D_2$, if $|V\ev{D_1}\cap V\ev{D_2}|\leq 1$ then $|E\ev{D_1}\cap E\ev{D_2}| = 0$, where the converse is not necessarily true. 
\end{remark}
\begin{definition}
Let $D_1$ and $D_2$ be two digraphs. The digraph $D_1 \ast D_2$ is defined as the union of $D_1$ and $D_2$, provided $| V\ev{D_1}\cap V\ev{D_2} | = 1$. 
\end{definition}
The digraph $D_1 \ast D_2$ is undefined, unless the intersection of $D_1$ and $D_2$ is a single vertex. Whenever defined, the binary operation $\ast$ is commutative and associative. We may omit parentheses when $\ast $ is applied more than once. 

\begin{definition}
\label{def:bridgeless_cactus}
Let $\mathcal{L}$ denote the set/collection of all finite cycles. Let $\mathcal{S}$ be the collection of digraphs, obtained by taking the closure of $\mathcal{L}$ under the operation $\ast$. 
\end{definition}

\begin{figure}
\centering
\includegraphics[width=0.65\linewidth]{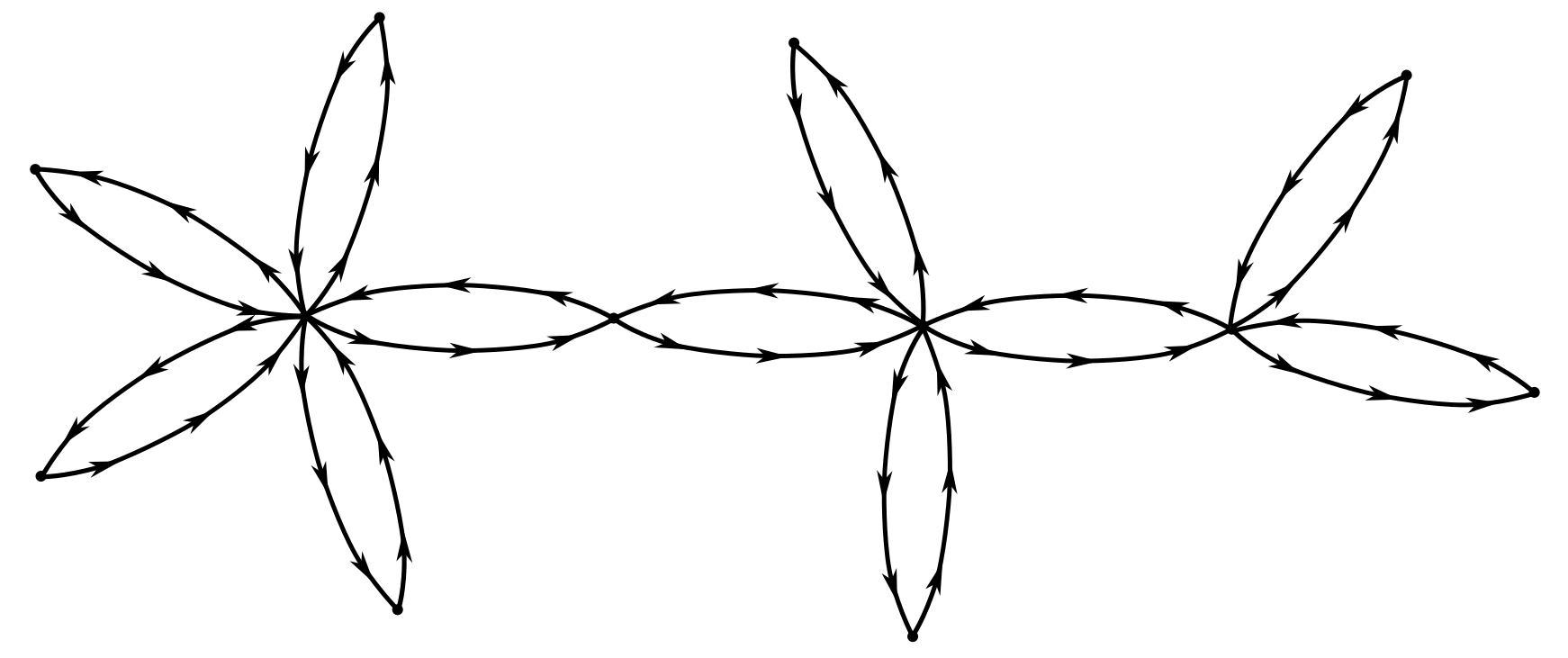}
\caption{An element of the collection $\mathcal{S}$.}
\label{fig:collectiontreelike}
\end{figure}

\begin{remark}
\label{rmk:basic_family}

\phantom{a}

\begin{enumerate}
\item An element of the collection $\mathcal{S}$ is similar to a tree (q.v. \Cref{fig:collectiontreelike}.) In fact, it is shown in \Cref{thm:unique_decomp_equiv_digraph} below that any $D\in \mathcal{S}$ has a unique arborescence rooted at any vertex. 
\item Given a digraph $D = \beta_1 \ast \ldots \ast \beta_t \in \mathcal{S}$, then it follows that $\beta_1\ast \ldots \ast \beta_i \in \mathcal{S}$, for each $i=1,\ldots,t$.
\item It is easy to see that $D_1 \ast D_2 \in \mathcal{S}$, for any two digraphs $D_1,D_2\in \mathcal{S}$ with $| V\ev{D_1}\cap V\ev{D_2} | = 1$.
\end{enumerate}
\end{remark}
\begin{definition}
Given a digraph $D$, a set $\mathcal{A} = \{ \gamma_1,\ldots, \gamma_t \} \subseteq \mathcal{B}\ev{D} $ is a \textit{partition of $D$ into cycles}, provided
$$
E\ev{D} = \bigsqcup_{i=1}^{t} E\ev{\gamma_i}
$$
is a disjoint union.
\end{definition}

\begin{remark}
\label{rmk:ast_directed_cycles}
Let $D_1$ and $D_2$ be two digraphs with $V\ev{ D_1 } \cap V\ev{ D_2 } = \{u\}$. Let $\beta \in \mathcal{B}\ev{ D_1 \ast D_2}$ be a cycle. Since $\beta$ visits the vertex $u$ at most once, it follows that $E\ev{ \beta }\subseteq E\ev{ D_1 }$ or $ E\ev{ \beta }\subseteq E\ev{ D_2 }$. Therefore, the set of cycles of $D_1\ast D_2$ is given by the disjoint union of the sets of cycles of $D_1$ and $D_2$:
$$ \mathcal{B}\ev{ D_1 \ast D_2 } = \mathcal{B}\ev{ D_1 } \sqcup \mathcal{B}\ev{ D_2 } $$
\end{remark}

\begin{lemma}
\label{lem:decomp_basic}
Let $D$ be an Eulerian digraph. Let $D_1$ and $D_2$ be Eulerian subdigraphs of $D$ such that $E\ev{D} = E\ev{D_1} \sqcup E\ev{D_2}$ is a disjoint union. Then,
\begin{enumerate}
\item[i)] If $\mathcal{A}_i$ is a partition of $D_i$ into cycles, for $i=1,2$, then $\mathcal{A}_1 \sqcup \mathcal{A}_2$ is a partition of $D$ into cycles. 
\item[ii)] If $D$ has a unique partition into cycles, then $D_i$ has a unique partition into cycles, for $i=1,2$. 
\end{enumerate}
\end{lemma}
\begin{proof}
\textit{i)}
$$ 
E\ev{D} = E\ev{D_1} \sqcup E\ev{D_2} = \ew{  \bigsqcup_{ \gamma \in \mathcal{A}_1 } E\ev{\gamma}  } \sqcup \ew{  \bigsqcup_{ \gamma \in \mathcal{A}_2 } E\ev{\gamma}  } =  \bigsqcup_{ \gamma \in \mathcal{A}_1\sqcup \mathcal{A}_2 } E\ev{\gamma} 
$$

\textit{ii)} Let $\mathcal{A}$ be the unique partition of $D$ into cycles. To show that $D_1$ has a unique partition, let $\mathcal{A}_1^{1}, \mathcal{A}_1^{2}$ be any two partitions of $D_1$ into cycles. Choose any partition $\mathcal{C}$ of $D_2$ into cycles. By Part \textit{i)}, we infer that $\mathcal{A}_1^{1} \sqcup \mathcal{C}$ and $\mathcal{A}_1^{2}\sqcup \mathcal{C}$ are both partitions of $D$ into cycles. By the uniqueness property of $\mathcal{A}$, we obtain $\mathcal{A} = \mathcal{A}_1^{1} \sqcup \mathcal{C} = \mathcal{A}_1^{2}\sqcup \mathcal{C}$. Since $E\ev{D_1}\cap E\ev{D_2} = \emptyset$, it follows that $\mathcal{A}_1^{1} = \mathcal{A}_1^{2}$, which proves the uniqueness of the partition of $D_1$ into cycles. By a similar argument, $D_2$ also has a unique partition.
\end{proof}
\begin{lemma}
\label{lem:ast_decomp}
Let $D_1$ and $D_2$ be Eulerian digraphs, intersecting at a singleton. Then, 
\begin{enumerate}
\item[i)] If $\mathcal{A}$ is a partition of $D = D_1 \ast D_2$ into cycles, then $ \mathcal{A}\cap \mathcal{B}\ev{ D_i }$ is a partition of $D_i$, for each $i=1,2$. 
\item[ii)] If $\mathcal{A}$ is the unique partition of the union $D = D_1 \ast D_2$, then $\mathcal{A}_i = \mathcal{A} \cap  \mathcal{B}\ev{D_i} $ is the unique partition of $D_i$, for each $i=1,2$. 
\item[iii)] If $D_i$ has a unique partition $\mathcal{A}_i$, for $i=1,2$, then $\mathcal{A}_1 \sqcup \mathcal{A}_2$ is the unique partition of the union $D_1 \ast D_2$.
\end{enumerate}
\end{lemma}

\begin{proof}
\textit{i)} Let $\mathcal{A}$ be any partition of $D = D_1 \ast D_2$ into cycles. Let $\mathcal{A}_i := \mathcal{A}\cap \mathcal{B}\ev{D_i}$. By \Cref{rmk:ast_directed_cycles}, we have 
$$ \mathcal{A} = \mathcal{A}\cap \mathcal{B}\ev{ D } = \mathcal{A}\cap (\mathcal{B}\ev{ D_1 } \sqcup \mathcal{B}\ev{ D_2 }) = ( \mathcal{A}\cap \mathcal{B}\ev{ D_1 }) \sqcup ( \mathcal{A}\cap \mathcal{B}\ev{ D_2 })  = \mathcal{A}_1 \sqcup \mathcal{A}_2 $$
This implies,
\begin{align*}
& E\ev{ D_1 } = E\ev{ D_1 } \cap E\ev{ D } \\ 
&  = E\ev{D_1} \cap \ew{ \bigsqcup_{\gamma \in \mathcal{A} } E\ev{ \gamma }  } && \text{ since $\mathcal{A}$ is a partition of the edges of $D$ into disjoint cycles} \\ 
& = \bigsqcup_{\gamma \in \mathcal{A} } (E\ev{ D_1 } \cap E\ev{ \gamma }) \\ 
& = \bigsqcup_{\gamma \in \mathcal{A}_1 } E\ev{ D_1 } \cap E\ev{ \gamma } && \text{ since $ E\ev{ D_1 } \cap E\ev{\gamma} =\emptyset$ for each $\gamma\in \mathcal{B}\ev{ D_2 }$. }\\ 
& = \bigsqcup_{\gamma \in \mathcal{A}_1 } E\ev{ \gamma }
\end{align*}
which means that $\mathcal{A}_1$ is a partition of $D_1$. Similarly, $\mathcal{A}_2$ is a partition of $D_2$. 

\textit{ii)} By Part \textit{i)}, $\mathcal{A}_i$ is a partition of $D_i$, for $i=1,2$. By \Cref{lem:decomp_basic} Part \textit{ii)}, we obtain that $\mathcal{A}_i$ is the unique partition of $D_i$ into cycles, for $i=1,2$.

\textit{iii)} Let $\mathcal{C}$ be any partition of $D = D_1 \ast D_2$. By Part \textit{i)}, $\mathcal{C}_i = \mathcal{C}\cap \mathcal{B}\ev{D_i}$ is a partition of $D_i$, for each $i=1,2$. By the uniqueness property of $\mathcal{A}_i$, we obtain $\mathcal{A}_i = \mathcal{C}_i$, for $i=1,2$. So, we obtain $\mathcal{C} = \mathcal{C}_1 \sqcup \mathcal{C}_2 = \mathcal{A}_1 \sqcup \mathcal{A}_2$.
\end{proof}

\begin{figure}[t!]
\centering
\includegraphics[width=5in]{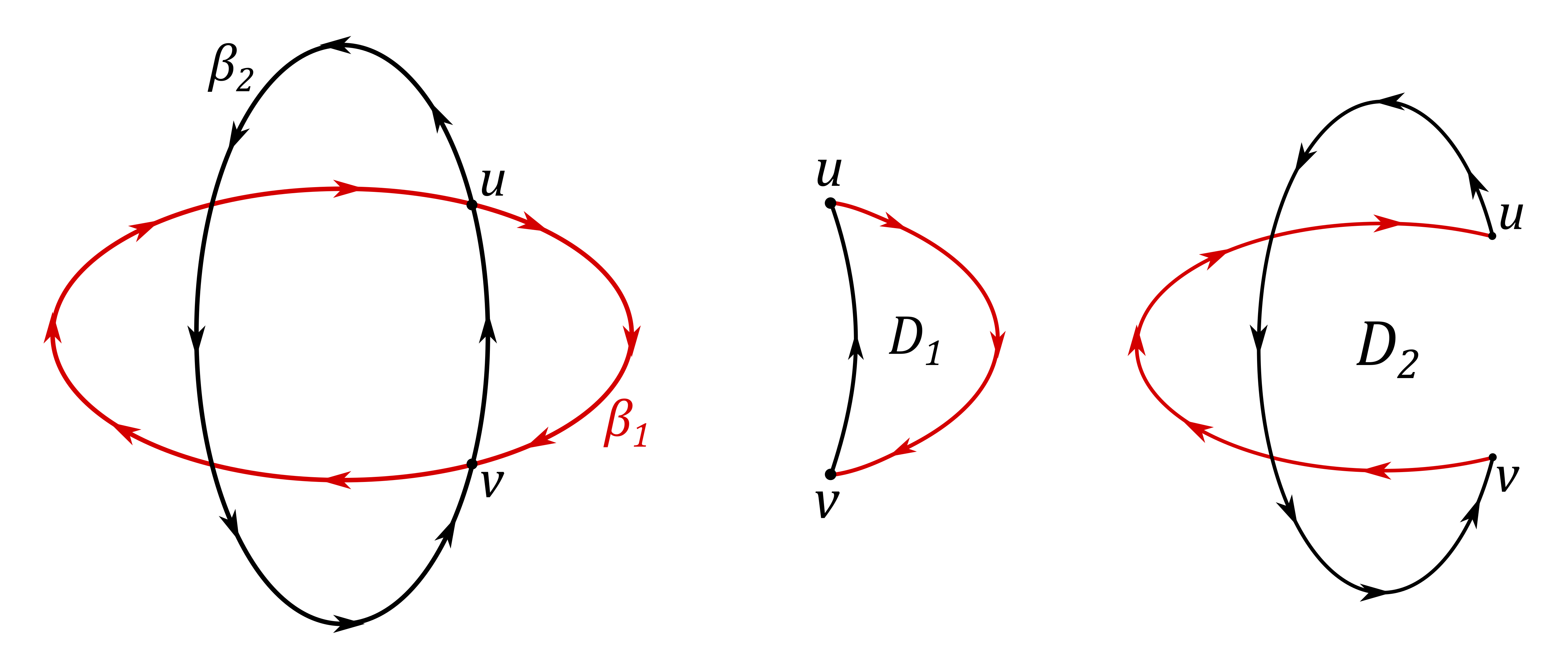} 
\caption{Two cycles $\{\beta_i\}_{i=1,2}$ intersecting at $u\neq v$ and the partitioning digraphs $\{D_i\}_{i=1,2}$}
\label{fig:uniquedecomponevertex}
\end{figure}

\begin{lemma}
\label{lem:max_intersection}
Let $D$ be a connected Eulerian digraph. Assume that $D$ has a unique partition $\mathcal{A} = \{ \beta_1, \ldots, \beta_t \}$ into cycles. Then, $ |V\ev{\beta_i} \cap V\ev{\beta_j} | \leq 1$, for each $1\leq i < j \leq t$. 
\end{lemma}
\begin{proof}
Suppose, for a contradiction, that there are some cycles in the partition of $D$, intersecting at more than one vertex. Relabeling if necessary, we may assume $ | V\ev{\beta_1} \cap V\ev{\beta_2} | \geq 2 $, so there are some $ u,v \in V\ev{\beta_1} \cap V\ev{\beta_2}$ such that $u \neq v$. Consider the digraphs (q.v. \Cref{fig:uniquedecomponevertex}), 
$$
D_1 :=  u \beta_1 v \beta_2 u  \text{ and } D_2 := u \beta_2 v \beta_1 u
$$ 
Note that $E\ev{ D_1} \cap E\ev{ D_2 } =\emptyset $ and $ E\ev{ D_1 } \sqcup E\ev{ D_2 } = E\ev{ \beta_1 } \sqcup E\ev{ \beta_2 }$. Let $\mathcal{C}_1 = \{ \gamma_1, \ldots, \gamma_m \}$ and $\mathcal{C}_2 = \{ \delta_1, \ldots, \delta_r \}$ (for $m,r\geq 1$) be any partition of $D_1$ and $D_2$ into cycles, respectively. We claim that,
\begin{equation}
\label{eqn:first_eq}
E\ev{ \epsilon } \cap E\ev{ \beta_1 } \neq \emptyset \text{ and } E\ev{ \epsilon } \cap E\ev{ \beta_2 } \neq \emptyset  \text{ for any $\epsilon \in \mathcal{C}_1\sqcup \mathcal{C}_2$ }
\end{equation}
To show this property, take any cycle $ \epsilon \in \mathcal{C}_1 $. We have $E\ev{ \epsilon }\subseteq E\ev{ D_1} \subseteq E\ev{ \beta_1 } \sqcup E\ev{ \beta_2 }$. Now, assume, for a contradiction, that $E\ev{ \epsilon }\subseteq E\ev{\beta_1}$. Since $\beta_1$ is a cycle, a proper subset of $E\ev{ \beta_1 }$ can not define a cycle. This implies that $\epsilon = \beta_1$. However, the path $ v \beta_1 u $ of $\beta_1$ is disjoint from $D_1$, which is a contradiction. Therefore, $E\ev{ \epsilon } \not \subseteq E\ev{ \beta_1 }$, which implies $E\ev{ \epsilon }\cap E\ev{ \beta_2 } \neq \emptyset$. By a similar argument, we obtain $E\ev{ \epsilon }\cap E\ev{ \beta_1 } \neq \emptyset$. 

By a similar argument, the property (\ref{eqn:first_eq}) holds for the cycles in $\mathcal{C}_2$, as well. Hence, we infer that the cycles of $\mathcal{C}_1\sqcup \mathcal{C}_2$ are distinct from $\beta_1 $ and $\beta_2$. Clearly, we also have $(\mathcal{C}_1 \sqcup \mathcal{C}_2) \cap \{\beta_3,\ldots,\beta_t\} = \emptyset$. It follows that the collection 
$$\mathcal{C}:= \mathcal{C}_1 \sqcup \mathcal{C}_2 \sqcup \{\beta_3, \ldots, \beta_t  \} = \{\gamma_1, \ldots, \gamma_m, \delta_1, \ldots, \delta_r, \beta_3, \ldots, \beta_t  \}$$ 
is a partition of $D$, distinct from $\mathcal{A}$, which is a contradiction. 
\end{proof}

\begin{lemma}
Let $D$ be a connected Eulerian digraph. Assume that $\mathcal{A} = \{ \beta_1, \ldots, \beta_t \}$ ($t\geq 2$) is the unique partition of $D$ into cycles. Assume that the cycles $\beta_1,\ldots,\beta_t$ intersect pairwise, i.e.,
$$ V\ev{ \beta_i } \cap V\ev{ \beta_j } \neq \emptyset$$
for each $1\leq i < j \leq t$. Then, $V\ev{ \beta_1 } \cap \ldots \cap V\ev{ \beta_t } $ is a singleton. In particular, $\{V\ev{ \beta_i } \}_{i=1}^{t}$ has the Helly property (\cite[p.~387, p.~393]{helly}).
\end{lemma}
\begin{proof}
By induction on $t\geq 2$. For the basis step, assume $t=2$. We have $V\ev{ \beta_1 } \cap V\ev{ \beta_2 } \neq \emptyset$, by hypothesis. Then, the result follows from \Cref{lem:max_intersection}. For the inductive step, let $\beta_1,\ldots,\beta_t$ be pairwise intersecting cycles, with $t\geq 3$. Since $\beta_1,\ldots,\beta_{t-1}$ intersect pairwise, the inductive hypothesis applies and so, we have $V\ev{ \beta_1 } \cap \ldots  \cap V\ev{ \beta_{t-1} } = \{u\}$, for some $u\in V\ev{D}$. Suppose, for a contradiction, that $u\notin V\ev{\beta_t}$. By hypothesis and \Cref{lem:max_intersection}, we have $ V\ev{ \beta_t } \cap V\ev{ \beta_i } = \{ v_i \}$, for $i=1,2$, where $v_i \neq u$. Consider the directed graphs (q.v. \Cref{fig:uniquedecomppairwiseintersecting}):
$$ \gamma_1 = u \beta_1 v_1 \beta_t v_2 \beta_2 u \text{ and } \gamma_2 = u \beta_2 v_2 \beta_t v_1 \beta_1 u  $$
Since the pairwise intersections of $\beta_1,\beta_2,\beta_t$ are singletons, it follows that $\gamma_{1}$ and $\gamma_2$ are cycles. Note that the edge set of $\beta_1$ is a disjoint union of two paths: $E\ev{\beta_1} = E\ev{ u\beta_1 v_1 } \sqcup E\ev{ v_1 \beta_1 u} $. The same applies to $\beta_2$ and $\beta_t$ and so, $E\ev{\gamma_1} \sqcup E\ev{\gamma_2} = E\ev{\beta_1} \sqcup E\ev{\beta_2} \sqcup E\ev{\beta_t}$. Therefore, the collection 
$$
\mathcal{C}:= \{ \gamma_1,\gamma_2, \beta_3, \ldots, \beta_{t-1}\}
$$ 
is a partition of $D$, distinct from $\mathcal{A}$, which is a contradiction. So, we must have $u\in V\ev{\beta_t}$, which implies, by \Cref{lem:max_intersection}, that $V\ev{\beta_t} \cap V\ev{ \beta_i } = \{u\}$ for each $i=1,\ldots,t-1$ and the result follows. 
\end{proof}

\begin{figure}[t!]
\centering
\includegraphics[width=5in]{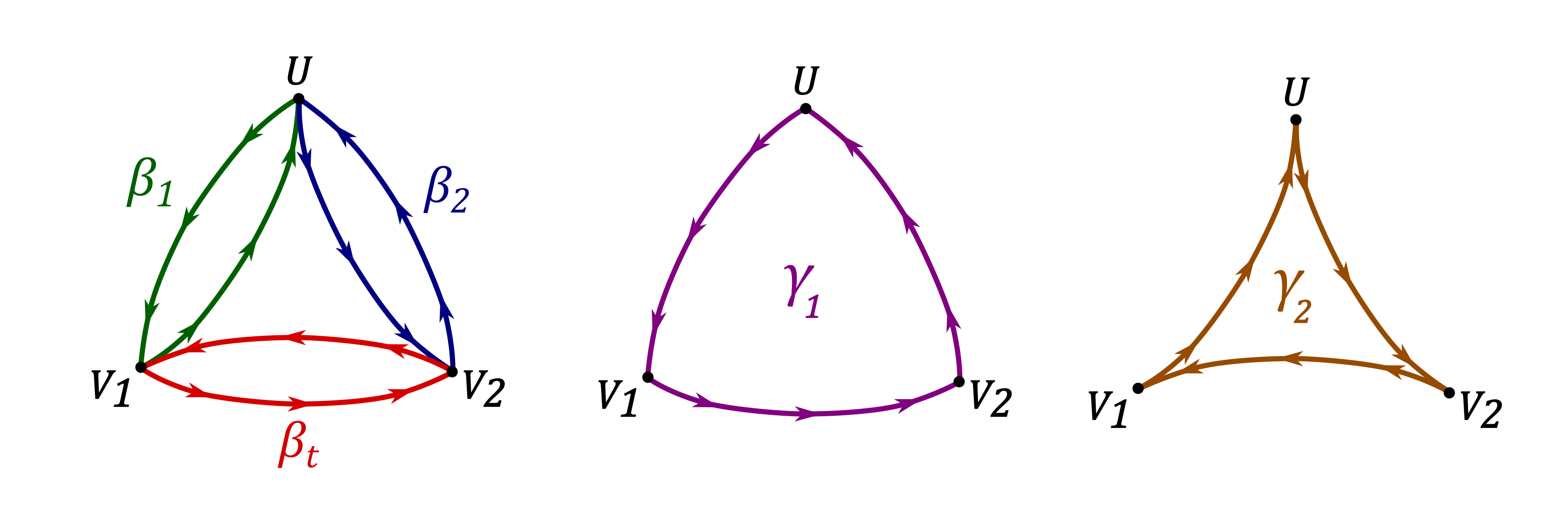} 
\caption{Pairwise intersecting cycles $\{\beta_1,\beta_2,\beta_t\}$ and the alternative partition}
\label{fig:uniquedecomppairwiseintersecting}
\end{figure}

\begin{lemma}
\label{lem:minimal_seq}
Given a digraph $D\in \mathcal{S}$, then, for each distinct $u,v \in V\ev{D}$, there is a sequence of cycles $\gamma_1,\ldots,\gamma_r$ of $D$ such that:
\begin{enumerate}
\item $V\ev{ \gamma_j } \cap V\ev{ \gamma_{j+1} } = \{ x_j \}$ is a singleton, for $j=1,\ldots,r-1$.
\item The vertices in $\{x_j\}_{j=1}^{r-1} \cup \{ u,v \} $ are pairwise distinct. 
\end{enumerate}
\end{lemma}
\begin{proof}
We apply induction on $t\geq 1$, to show that a digraph of the form $D = \beta_1 \ast \ldots \ast \beta_t \in \mathcal{S}$, where $\{\beta_i\}_{i=1}^{t}$ are cycles, satisfies the listed conditions. For the base case, let $t=1$. Then, $D = \beta_1$ and so, $\gamma_1:=\beta_1$ satisfies the conditions. For the inductive step, let $D = \beta_1 \ast \ldots \ast \beta_t \in \mathcal{S}$ be given. Define $D' := \beta_1 \ast \ldots \ast \beta_{t-1}$, where $V\ev{D'}\cap V\ev{\beta_t} = \{y\}$ for some $y$. Note that $D'\in \mathcal{S}$, by \Cref{rmk:basic_family} Part 2. We consider several cases:

Case 1: $u,v\in V\ev{D'}$. By the inductive hypothesis applied to $D'$ and $u,v\in V\ev{D'}$, we obtain the sequence $\{\gamma_{j}\}_{j=1}^{r}$ in question. 

Case 2: $u,v\not\in V\ev{D'}$. Necessarily, we have $u,v\in V\ev{\beta_t}$, so the sequence $\gamma_1:=\beta_t$ works, as in the base case. 

Case 3: $u\in V\ev{D'}$ and $v\not\in V\ev{D'}$. Then, by the inductive hypothesis applied to $D'$ and $u,y\in V\ev{D'}$, we obtain a sequence $\{\gamma_{j}\}_{j=1}^{r}$, such that $V\ev{ \gamma_j } \cap V\ev{ \gamma_{j+1} } = \{ x_j \}$ is a singleton, for $j=1,\ldots,r-1$ and the vertices in $\{x_j\}_{j=1}^{r-1} \cup \{ u,y\} $ are pairwise distinct. We extend the sequence by defining $\gamma_{r+1} := \beta_t$ and $x_r := y$. Since $v\notin V\ev{D'}$, we have $v\in V\ev{\beta_t} \setminus \{y\} $. Also, note that the vertices in $\{ x_j\}_{j=1}^{r-1} \cup \{y \} \cup \{u,v\}$ are pairwise distinct. Finally, $V\ev{ \beta_t } \cap V\ev{ \gamma_r } = \{y\} $, so the statement follows. 

Case 4: $v\in V\ev{D'}$ and $u\not\in V\ev{D'}$. The proof is identical with Case 3, after interchanging the roles of $u$ and $v$. 
\end{proof}

\begin{lemma}
\label{lem:unique_decomp_equiv_treelike}
Let $D$ be a connected Eulerian digraph. Then, $D$ has a unique partition into cycles if and only if $D\in \mathcal{S}$. 
\end{lemma}
\begin{proof}
($\Longrightarrow$) By induction on the number of cycles in the unique partition of $D$. In the basis step, we have a single cycle $D = \beta_1 \in \mathcal{L}$, which implies $D\in \mathcal{S}$. For the inductive step, let $D$ be a digraph with a unique partition $\mathcal{A} = \{\beta_1,\ldots,\beta_t\}$ ($t\geq 2$) into cycles. Let $D_1, \ldots, D_m$ (for $m\geq 1$) be the weakly connected components of the subgraph $D \setminus E\ev{\beta_t}$. By \Cref{lem:decomp_basic} Part \textit{ii)}, the digraph $D_i$ has a unique partition into cycles. So, by the inductive hypothesis, we have $D_i\in \mathcal{S}$. We claim that $ | V\ev{\beta_t} \cap V\ev{ D_i } | = 1$, for each $i=1,\ldots,m$. Fix any $i$ and assume, by way of contradiction, that $u,v \in V\ev{\beta_t} \cap V\ev{ D_i } $ with $u\neq v$ (cf. \Cref{fig:uniquedecompifftreelike}). By \Cref{lem:minimal_seq}, there is a sequence of cycles $ \{  \gamma_j \}_{j=1}^{r}\subseteq \mathcal{B}\ev{D_i} \subseteq \mathcal{B}\ev{D}$ such that $V\ev{ \gamma_j } \cap V\ev{ \gamma_{j+1} } = \{ x_j \}$ is a singleton, for $j=1,\ldots,r-1$ and the vertices $\{x_j\}_{j=1}^{r-1} \cup \{ u,v \} $ are pairwise distinct. Now, we can obtain an alternative partition of $D$ into cycles as follows: Define the cycles,
$$\delta_{1}:= u \ \gamma_1 \ x_1\ \gamma_2\ x_2\ \ldots\ x_{r-1}\ \gamma_r\ v\ \beta_t \ u \text{ and } \delta_{2}:=  u\ \beta_t \ v\ \gamma_r\ x_{r-1}\ \gamma_{r-1}\ x_{r-2}\ \ldots\ \ x_2\ \gamma_2\ x_1\ \gamma_1\ u $$
It is easy to check that the collection 
$$\ev{ \{\beta_1,\ldots,\beta_t\} \setminus \{ \gamma_1,\ldots,\gamma_r, \beta_t\} } \cup \{ \delta_1,\delta_2\}$$
is a partition of $D$ into cycles, which contradicts the assumption that $D$ has a unique partition. Hence, we obtain $| V\ev{\beta_t} \cap V\ev{ D_i } | \leq 1$, for each $i=1,\ldots,m$. As $V\ev{D_i}\cap V\ev{\beta_t} \neq \emptyset$, we have $| V\ev{\beta_t} \cap V\ev{ D_i } | = 1$. The Eulerian digraphs $\{D_i\}_{i=1}^{m}$ are pairwise disjoint and so, by \Cref{rmk:basic_family} Part 3, we obtain $D = \beta_t \ast D_1 \ast \ldots \ast D_m \in \mathcal{S}$. 

($\Longleftarrow$) We apply induction on $t\geq 1$, to show that a digraph of the form $D = \beta_1 \ast \ldots \ast \beta_t \in \mathcal{S}$, where $\{\beta_i\}_{i=1}^{t}$ are cycles, has unique partition $\{\beta_1,\ldots,\beta_t\}$. The basis step, $t=1$ is clear. For the inductive step, let $D = \beta_1 \ast \ldots \ast \beta_t \in \mathcal{S}$ be given, where $t\geq 2$. By \Cref{rmk:basic_family} Part 2, we have $\beta_1 \ast \ldots \ast \beta_{t-1}\in \mathcal{S}$. By the inductive hypothesis, $\{\beta_1,\ldots, \beta_{t-1} \}$ is the unique partition of $\beta_1 \ast \ldots \ast \beta_{t-1}$. By \Cref{lem:ast_decomp} Part \textit{iii)}, this implies that $\{\beta_1,\ldots, \beta_{t}\} = \{\beta_1,\ldots, \beta_{t-1}\} \sqcup \{ \beta_t \}$ is the unique partition of $D = \ev{ \beta_1 \ast \ldots \ast \beta_{t-1} } \ast \beta_t$ into cycles. 
\end{proof}

\begin{figure}[t!]
\centering
\includegraphics[width=3.6in]{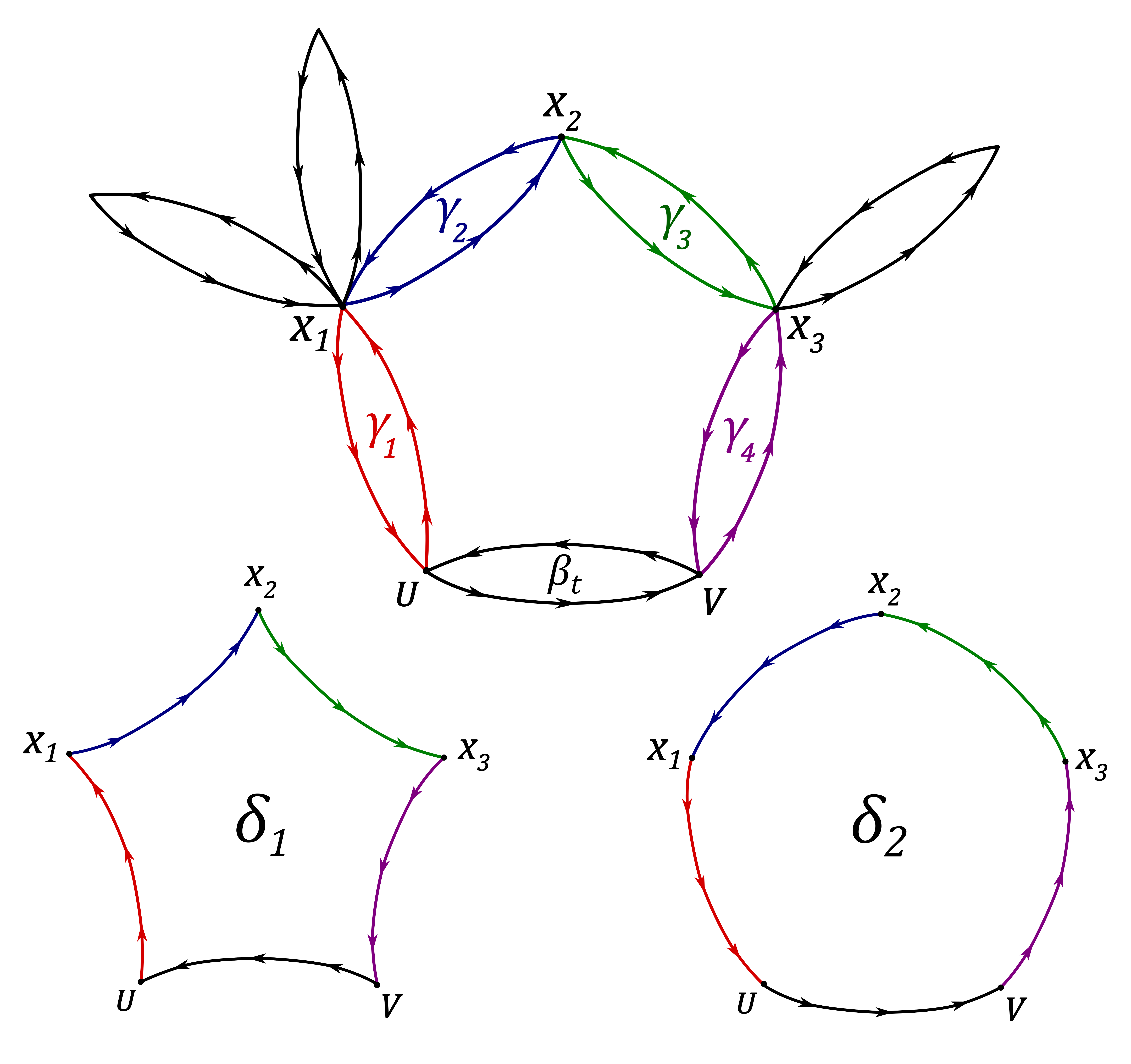} 
\caption{Alternative partition of $D_i \cup \beta_t$ in a special case of the proof of \Cref{lem:unique_decomp_equiv_treelike}}
\label{fig:uniquedecompifftreelike}
\end{figure}


Now, we list many conditions equivalent to having a unique partition into cycles, for a connected Eulerian digraph. First, let us recall some preliminary definitions. 

It is well known that, for a connected digraph $D$, if $D$ is Eulerian, then $D$ is \textit{strongly connected}, i.e., there is a directed path between every distinct pair of vertices. Furthermore, strongly connected digraphs are bridgeless. Also, $D$ is bridgeless if and only if every edge of $D$ is contained in at least one cycle. 

Given a connected digraph $D$, then consider the partition lattice $\Pi\ev{D}$ of the edge set $E\ev{D}$. The subposet of $\Pi\ev{D}$ induced by partitions into Eulerian (and connected) parts is denoted by $T\ev{D}$. It is previously shown (\cite[Lemma~3.1, p.~16]{co3}) that $T\ev{D}$ is a join-semilattice and it is a lattice if and only if $D$ has a unique partition into cycles. 

We obtain the following characterizing theorem for connected Eulerian digraphs with a unique partition into cycles.

\begin{theorem}
\label{thm:unique_decomp_equiv_digraph}
Given a connected Eulerian digraph $D$, then the following are equivalent:
\begin{enumerate}
\item[\textit{(1)}]	No two cycles of $D$ share an edge, i.e., $D$ is a (bridgeless) cactus digraph.
\item[\textit{(2)}] Every edge of $D$ is contained in exactly one cycle. 
\item[\textit{(3)}] $\mathcal{B}\ev{D} $ is a partition of $D$ into cycles.
\item[\textit{(4)}] $D$ has a unique partition into cycles.
\item[\textit{(5)}] $D \in \mathcal{S}$. 
\item[\textit{(6)}] $|\mathcal{B}\ev{D}| = \mathbb{C}\ev{D}  $, where $\mathbb{C}\ev{D} := |E\ev{D}|-|V\ev{D}|+1$.
\item[\textit{(7)}] $\tau_D = 1$, i.e., $D$ has a unique arborescence rooted at $u$, for any $u\in V\ev{D}$.
\item[\textit{(8)}] $T\ev{D}$ is a lattice and in particular has a unique minimal element.
\end{enumerate}
\end{theorem}
\begin{proof}

\phantom{a}

The equivalences \textit{(1)}$\Longleftrightarrow$\textit{(2)} and \textit{(1)}$\Longleftrightarrow$\textit{(3)} are clear. 

\textit{(3)} $\Longrightarrow$ \textit{(4)}: By hypothesis, we have a disjoint union 
$$
E\ev{D} = \bigsqcup_{\gamma\in \mathcal{B}\ev{D} } E\ev{ \gamma }
$$
In particular, we have 
\begin{equation}
\label{eqn:gammas_distinct}
E\ev{ \gamma_1 } \cap E\ev{ \gamma_2 } = \emptyset \qquad \text{  for any two distinct $\gamma_1, \gamma_2\in \mathcal{B}\ev{D}$}
\end{equation}
Let $ \mathcal{A} = \{ \beta_1,\ldots,\beta_t\}$ be any partition of $D$ into cycles. We claim that $\mathcal{A} = \mathcal{B}\ev{D} $. By the definition of $\mathcal{B}\ev{D}$, we have $\mathcal{A}\subseteq \mathcal{B}\ev{D}$. Conversely, given $\gamma\in \mathcal{B}\ev{D}$, we have
$$ 
\emptyset \neq E\ev{ \gamma } \cap E\ev{D} = E\ev{ \gamma } \cap \ew{ \bigsqcup_{i=1}^{t} E\ev{ \beta_i }  }= \bigsqcup_{i=1}^{t}  \ev{ E\ev{ \gamma } \cap E\ev{ \beta_i } }
$$
So, there is at least one $\beta_i$ such that $E\ev{ \beta_i }\cap E\ev{ \gamma } \neq \emptyset$. By (\ref{eqn:gammas_distinct}), we obtain $\gamma = \beta_i \in \mathcal{A}$. 

\textit{(4)} $\Longleftrightarrow$ \textit{(5)}: By \Cref{lem:unique_decomp_equiv_treelike}.

\textit{(5)} $\Longrightarrow$ \textit{(1)}, \textit{(5)} $\Longrightarrow$ \textit{(6)}, \textit{(5)} $\Longrightarrow$ \textit{(7)}: By induction on the number of cycles. 

\textit{(6)} $\Longrightarrow$ \textit{(5)}: By induction on the number of edges, we show that for any Eulerian digraph $D=\ev{V\ev{D},E\ev{D},\psi}$, we have
\begin{align}
& |\mathcal{B}\ev{D}|\geq \mathbb{C}\ev{D} \label{eq:1} \\ 
& |\mathcal{B}\ev{D}| = \mathbb{C}\ev{D} \text{ implies } D \in \mathcal{S} \label{eq:2}
\end{align}
The base case of $|E\ev{D}| = 2$ is clear. For the inductive step, let $E\ev{u,v} := \{ e\in E\ev{D} : \psi\ev{e} = \ev{u,v}\}$ for any pair of distinct vertices $u,v \in V\ev{D}$. We have four cases to consider. 

\textbf{Case 1:} If there is some pair $u,v$ such that $|E\ev{u,v}| = 1$ and $|E\ev{v,u}|=0$, then let $e$ be the unique edge such that $\psi\ev{e} = \ev{u,v}$. Consider the quotient digraph $D'$ obtained by contracting $e$ (q.v. \Cref{def:edge_contract}). Then, we have $ | \mathcal{B}\ev{D'} | = |\mathcal{B}\ev{D}|$ and $ \mathbb{C}\ev{D'} = \mathbb{C}\ev{D} $. To show (\ref{eq:2}), assume $|\mathcal{B}\ev{D}| = \mathbb{C}\ev{D}$. By the inductive hypothesis, we obtain $D'\in \mathcal{S}$. Then, $D\in \mathcal{S}$ follows. 

\textbf{Case 2:} If there is a pair $u,v$ such that $|E\ev{u,v}| = 1$ and $|E\ev{v,u}|=1$, then there is an edge $e\in E\ev{D}$ with $\psi\ev{e} = \ev{u,v}$. Let $D'$ be the quotient digraph obtained by contracting $e$ (q.v. \Cref{def:edge_contract}). It is easy to see that $|\mathcal{B}\ev{D'}| = |\mathcal{B}\ev{D}| - 1$ and $\mathbb{C}\ev{D'} = \mathbb{C}\ev{D} - 1$. As in Case 1, the inductive hypothesis applies and the statements follow.

\textbf{Case 3:} For each pair $u,v$, we have $|E\ev{u,v}| = 0 $ and $|E\ev{v,u}|\geq 2$ or $|E\ev{u,v}|\geq 2$ and $|E\ev{v,u}| = 0 $. Then, removing a cycle does not disconnect $D$. So, let $\beta \in \mathcal{B}\ev{D}$ be any cycle, of length $\ell \geq 2$ and $D'$ be the Eulerian subdigraph with the same vertex set and edge set $E\ev{D} \setminus E\ev{\beta}$. We claim that $| \mathcal{B}\ev{D} | > \mathbb{C}\ev{D}$. Note that $D'$ is Eulerian, $\mathbb{C}\ev{D} = \mathbb{C}\ev{D'} + \ell$ and $ | \mathcal{B}\ev{D}| \geq | \mathcal{B}\ev{D'} | + 2^\ell - 1 $. In particular, by the inductive hypothesis, we obtain 
$$
| \mathcal{B}\ev{D} | \geq | \mathcal{B}\ev{D'} | + 2^\ell - 1 > \mathbb{C}\ev{D'} + \ell = \mathbb{C}\ev{D} 
$$ 
and both claims follow.

\textbf{Case 4:} For each pair $u,v$, we have $|E\ev{u,v}| \geq 1 $ and $|E\ev{v,u}|\geq 2$ or $|E\ev{u,v}|\geq 2$ and $|E\ev{v,u}| \geq 1 $. Let $e,f$ be fixed such that $\psi\ev{e} =\ev{u,v}$ and $\psi\ev{f} = \ev{v,u}$. Let $D' $ be the subdigraph with the same vertex set and edge set $E\ev{D}\setminus \{e,f\}$. Again, we claim that $| \mathcal{B}\ev{D} | > \mathbb{C}\ev{D} $. Note that $D'$ is Eulerian, $\mathbb{C}\ev{D} = \mathbb{C}\ev{D'} + 2$ and $ | \mathcal{B}\ev{D}| \geq | \mathcal{B}\ev{D'} | + 2 $.  First, it follows from the first part of the inductive hypothesis that
$$
| \mathcal{B}\ev{D} | \geq | \mathcal{B}\ev{D'} | + 2 \geq \mathbb{C}\ev{D'} + 2 = \mathbb{C}\ev{D} 
$$
Suppose, for a contradiction, that $| \mathcal{B}\ev{D} | = | \mathbb{C}\ev{D}  |$. Then, the above inequalities are equalities and in particular, $| \mathcal{B}\ev{D} | = | \mathcal{B}\ev{D'} | + 2$. By the second part of the inductive hypothesis, we obtain $D'\in \mathcal{S}$. Then, there is a unique cycle $\beta$ of $D'$ containing $u$ and $v$. However, in this case we have at least three distinct cycles in $\mathcal{B}\ev{D} \setminus \mathcal{B}\ev{D'}$, namely: 
$$
[u\beta v f u]_{\rho}, \ [v\beta u e v]_{\rho} ,\ [u e v f u]_{\rho} 
$$
which is a contradiction.

\textit{(7)} $\Longrightarrow$ \textit{(5)}: By induction on the number of edges. The base case $|E\ev{D}|=2$ is clear, as there is a unique Eulerian digraph on $2$ edges. For the inductive step, let $e\in E\ev{D}$ be a fixed edge, with $\psi\ev{e} = \ev{u,v}$. Let $T \in \tau^{u}\ev{D}$ be the unique arborescence of $D$, rooted at $u$. First, we claim that $e\in E\ev{T}$. Suppose, to the contrary, that $e\notin E\ev{T}$. Let $f\in E\ev{T}$ be the unique edge of $T$ with terminal vertex $v$. Then, we obtain a distinct arborescence $T' := (T \setminus \{f\}) \cup \{e\} \in \tau^{u}\ev{D}$, which is a contradiction. Let $D'$ be the Eulerian digraph obtained by contracting $e$ (q.v. \Cref{def:edge_contract}). By \Cref{rmk:edge_contract_arb_num}, we have $1 = \tau_D = | \tau^{u}\ev{D} | = | \tau^{u}_e\ev{D} | = | \tau^{u'} \ev{ D' } |$. By the inductive hypothesis, $D'\in \mathcal{S}$. It follows that $D \in \mathcal{S}$. 

\textit{(4)} $\Longleftrightarrow$ \textit{(8)}: By \cite[Lemma~3.1, p.~16]{co3}.
\end{proof}

By the preceding theorem, whenever a digraph has a unique partition into cycles, we can deduce that the unique partition must use all cycles of $D$. We finish this section by noting that yet another condition equivalent to $| \mathcal{B}\ev{D} | = \mathbb{C}\ev{D}$ is given in \cite[p.~132, Exercise 5.4.3.(b)]{bondy}.

\subsection{Digraphs with a unique Eulerian circuit}
\label{sec:unique_eulerian}
A special type of a connected Eulerian digraph is one with a unique Eulerian circuit, i.e. $|\mathcal{C}\ev{D}|=1$. A sufficient condition for having a unique Eulerian circuit is given by Arratia, Bollob\'{a}s, and Sorkin (\cite[p.~202, Lemma 3]{bollobas}), for a \textit{2-in, 2-out} digraph, i.e., an Eulerian digraph $D$, where each vertex has in-degree and out-degree $2$, possibly with loops. In \Cref{prop:unique_eulerian}, we obtain a characterization for connected loopless Eulerian digraphs, without imposing the condition of 2-in, 2-out. Afterwards, in \Cref{rmk:vacuous}, we generalize the result to any digraph, possibly with loops.

\begin{definition}
Given a digraph $D$ and a closed trail 
$$
\mathbf{w} = \ev{v_0,e_1,v_1,e_2,\ldots, v_{d-1}, e_{d}, v_d} \in \mathcal{T}\ev{D}
$$
then a pair $\ev{a,b} \in V\ev{D}\times V\ev{D}$ of vertices is an \textit{interlacing pair} of $\mathbf{w}$, provided $\mathbf{w}$ visits them in the order $a,b,a,b$ at least once, in other words, there are indices $0\leq i_1 < i_2 < i_3 < i_4 \leq d $ such that $v_{i_1} = v_{i_3} = a$ and $v_{i_2} = v_{i_4} = b$ and $v_j \in V\ev{D} \setminus \{a,b\}$ for $j \in [i_1,i_4]\setminus \{ i_2,i_3 \}$. The pair $\ev{a,b}$ is an interlacing pair of an Eulerian circuit $z\in \mathfrak{C}\ev{D}$, provided there is some $\mathbf{w}$ such that $z=[\mathbf{w}]_{\rho}$ and $\ev{a,b}$ is an interlacing pair of $\mathbf{w}$.
\end{definition}
It follows from the definition that if $\ev{a,b}$ is an interlacing pair of an Eulerian circuit $z$, then $\ev{b,a}$ is also an interlacing pair of $z$. Also, it is easy to see that for a vertex $a\in V\ev{D}$ and an Eulerian circuit $z\in \mathfrak{C}\ev{D}$,
\begin{equation}
\label{eq:3}
\text{ $\ev{a,a}$ is an interlacing pair of $z$ if and only if the out-degree of $a$ is at least $3$. }
\end{equation}

Recall that a subdigraph of a connected digraph $D$ is \textit{biconnected}, provided the removal of a vertex does not disconnect it. A \textit{block} or \textit{biconnected component} in a digraph is a biconnected subdigraph that is maximal with respect to inclusion of vertices. A cactus digraph is a \textit{Christmas} cactus digraph, provided each vertex is contained in at most two blocks. Note that the blocks of a bridgeless cactus digraph $D$ are exactly the cycles of $D$. 

Given a set of cycles $\mathcal{A}$, then the \textit{intersection graph} of $\mathcal{A}$, denoted $G_\mathcal{A}$, is the graph with vertex set $\mathcal{A}$ and there is an edge between two vertices $\beta_1,\beta_2 \in \mathcal{A}$ if and only if $V\ev{\beta_1} \cap V\ev{\beta_2} \neq \emptyset$ (c.f. \cite[p.~1]{mckee}). It is easily seen that $D$ is a bridgeless Christmas cactus digraph if and only if $G_{\mathcal{B}\ev{D}}$ is a tree. 

We are ready to list conditions equivalent to having a unique Eulerian circuit, for a digraph $D$. The equivalence \textit{(3)}$\Longleftrightarrow$\textit{(4)} is stated in \cite[p.~67]{pevzner} and shown in \cite[p.~152, Theorem 7.5]{waterman} (also see \cite{acosta}). (The following proof provides an alternative argument.)

\begin{proposition}
\label{prop:unique_eulerian}
Given a digraph $D$, then the following are equivalent: 
\begin{enumerate}
\item[\textit{(1)}] $D \in \mathcal{S} $ ($D$ is a bridgeless cactus digraph) and $\Delta^{+}\ev{D} \leq 2$, where $\Delta^{+}\ev{D}$ is the maximum out-degree of the vertices of $D$. 
\item[\textit{(2)}] $D \in \mathcal{S}$ ($D$ is a bridgeless cactus digraph) and each vertex is contained in at most two blocks; in other words, $D$ is a bridgeless Christmas cactus digraph.
\item[\textit{(3)}] The intersection graph $G_{\mathcal{B}\ev{D}}$ is a tree. 
\item[\textit{(4)}] $D$ has a unique Eulerian circuit. 
\item[\textit{(5)}] $D$ has an Eulerian circuit with no interlacing pair. 
\end{enumerate}
\end{proposition}

\begin{proof}
$\neg$\textit{(2)} $\Longrightarrow$ $\neg$\textit{(1)}:  Assume that $D$ has a vertex $v$ contained in at least three blocks, i.e., cycles of $D$. Then, the out-degree of $v$ is at least three, since $D$ has no two cycles sharing an edge, which is a contradiction. 

\textit{(2)} $\Longrightarrow$ \textit{(1)}: Assume \textit{(2)} and suppose, by way of contradiction, that $D$ has a vertex $v$ of out-degree $k \geq 3$. Choose three distinct edges $\{e_i\}_{i=1}^{3}$ with initial vertex $v$. Then, as $D$ is bridgeless, there are at least three cycles $\{ \beta_i \}_{i=1}^{3}$ such that $\beta_i$ contains $e_i$, for each $i$. As $D$ is a cactus digraph, the cycles $\{\beta_i\}_{i=1}^{3}$ are three distinct blocks containing $v$, a contradiction. 	

\textit{(2)} $\Longleftrightarrow$ \textit{(3)}: The proof is similar to that of \textit{(1)} $\Longleftrightarrow$ \textit{(2)}. 

\textit{(1)} $\Longrightarrow$ \textit{(4)}: By induction on the number $t$ of cycles of $D$. The base case $t=1$ is clear. For the inductive step, let $D = \beta_1\ast \ldots \ast \beta_t \in \mathcal{S}$ and $D' := \beta_1 \ast \ldots \ast \beta_{t-1}$ and fix an edge $e\in E\ev{D'}$. By the inductive hypothesis, $D'$ has a unique Eulerian circuit and so, it has a unique Eulerian trail $\mathbf{w}$ ending at $e$. If $D$ has two distinct Eulerian trails $\mathbf{w}_1$ and $\mathbf{w}_2$ ending at $e$, then we obtain distinct Eulerian trails $\mathbf{w}_i \setminus \beta_t$ ($i=1,2$) of $D'$, which is a contradiction. 

\textit{(4)} $\Longrightarrow$ \textit{(1)}: By the B.E.S.T. Theorem (q.v. \Cref{thm:best}), we obtain $\tau_D = 1$ and $\Delta^{+}\ev{D} \leq 2$. The statement follows from \Cref{thm:unique_decomp_equiv_digraph}.

\textit{(1)} $\Longrightarrow$ \textit{(5)}: We proceed by induction on the number of cycles. Let $D = \beta_1\ast \ldots \ast \beta_t \in \mathcal{S}$ be a bridgeless cactus digraph with maximum out-degree $2$. The base case, $t=1$ is clear. For $t\geq 2$, let $D' := \beta_1 \ast \ldots \ast \beta_{t-1}$ and $V\ev{D'}\cap V\ev{\beta_t} = \{a\}$. By the implication \textit{(1)} $\Longrightarrow$ \textit{(4)}, $D$ has a unique Eulerian circuit $z\in \mathfrak{C}\ev{D}$. Let $\mathbf{w} = \ev{v_0,e_1,v_1,e_2,\ldots, v_{d-1}, e_d, v_d}$ be any Eulerian closed trail such that $z = [\mathbf{w}]_{\rho}$. If $v_0\in V\ev{D'}$, then $\mathbf{w}':= \mathbf{w} \setminus \beta_t $ is an Eulerian closed trail of $D'$. If, on the other hand, $v_0 \in V\ev{\beta_t}\setminus V\ev{D'}$, then we have $v_i = v_j = a $ for some $0 < i < j < d$, as the out-degree of $a$ is $2$, and so $\mathbf{w}':= v_i \mathbf{w} v_j$ is an Eulerian closed trail of $D'$ and $
\mathbf{w} = (v_0 \mathbf{w} v_i) \mathbf{w}' (v_j\mathbf{w}v_d)$. In both cases, $\mathbf{w}'$ has no interlacing pair by the inductive hypothesis. The vertices $V\ev{\beta_t}\setminus \{a\}$ are visited exactly once by $\mathbf{w}$ and so they can not participate in an interlacing pair. Finally, we infer by (\ref{eq:3}) that $\ev{a,a}$ is not an interlacing pair of $z$. 

\textit{(5)} $\Longrightarrow$ \textit{(1)}: We apply induction on the number of edges. The base case is clear. In the inductive step, let $z$ be an Eulerian circuit of $D$ with no interlacing pairs. Fix an edge $e\in E\ev{D}$. Then, $z = [\mathbf{w}]_{\rho}$ for some Eulerian closed trail 
$$ 
\mathbf{w} = \ev{v_0,e_1,v_1,e_2,\ldots, v_{d-1}, e, v_d}
$$
ending at $e$ and with no interlacing pair. Let $\mathbf{w}_1 := v_i \mathbf{w} v_j$ be the first simple closed trail of $\mathbf{w}$. Define the cycle $\beta := [\mathbf{w}_1]_{\rho}$ and the subdigraph $D'$ with edge set $E\ev{D}\setminus E\ev{\beta}$. Then, $\mathbf{w} \setminus \mathbf{w}_1$ is an Eulerian closed trail of $D'$ with no interlacing pair. By the inductive hypothesis, $D' = \beta_1\ast \ldots \ast \beta_t \in \mathcal{S}$ and $\Delta^{+}\ev{D'} \leq 2$. We claim that $V\ev{D'} \cap V\ev{\beta} = \{v_i\}$. Suppose, for a contradiction, that there is some index $k$ such that $ i < k < j $ and $v_k \in (V\ev{D'} \cap V\ev{\beta}) \setminus \{ v_i \}$. As in the proof of the implication \textit{(4)} $\Longrightarrow$ \textit{(1)}, we have $v_k = v_r$, for some $r\in [j+1,d]$. This leads to a contradiction, as $\ev{v_i,v_k}$ forms an interlacing pair. Therefore, we obtain $D = D' \ast \beta \in \mathcal{S}$. Finally, it follows from (\ref{eq:3}) that the out-degree of $v_i$ is at most $2$. 
\end{proof}

By the preceding result, any loopless digraph with a unique Eulerian circuit is a Christmas cactus, and in particular has a vertex of out-degree $1$. Therefore, a loopless $2$-in, $2$-out digraph $D$ has more than one Eulerian circuit, and any Eulerian circuit of $D$ admits an interlacing pair. 

\begin{remark}
\label{rmk:vacuous}
We can see that previous theory easily generalizes to digraphs with loops. A digraph $D$ is \textit{loop-allowed}, provided there are $d_v\geq 0$ loops at each vertex $v\in V\ev{D}$. If $d_v = 0$ for all $v$, then $D$ is called \textit{loopless} (loopless digraphs are loop-allowed). All the preceding definitions naturally generalize to loop-allowed digraphs. We use the same notation $\mathcal{S}$ for the family of bridgeless cactus loop-allowed digraphs. For a loop-allowed digraph $D$, let $\underline{D}$ be the loopless digraph on $|V\ev{D}| + \sum_{v\in V\ev{D}} d_v$ vertices, obtained by replacing each loop with a digon. Then, we note that blocks (i.e., cycles) of $\underline{D}$ and $D$ are in one-to-one correspondence, $\underline{D} \in \mathcal{S}$ if and only if $D \in \mathcal{S}$, $\Delta^{+}\ev{ \underline{D} } = \Delta^{+}\ev{D}$, $\mathbb{C}\ev{ \underline{D} }  = \mathbb{C}\ev{D} $, a circuit $z\in \mathfrak{C}\ev{\underline{D} } $ has an interlacing pair if and only if the corresponding circuit in $D$ has an interlacing pair, and the posets $ T\ev{\underline{D}}$ and $ T\ev{D}$ are isomorphic. It follows that both \Cref{thm:unique_decomp_equiv_digraph} and \Cref{prop:unique_eulerian} generalize with no changes to digraphs possibly with loops. In particular, we have a generalization of \cite[p.~202, Lemma 3]{bollobas}, where the condition of $2$-in, $2$-out is removed and the converse statement is supplied. 
\end{remark}

\begin{remark}
As an application, consider the De Bruijn digraph $D_{n}$ (\cite{li_zhang}), for $n \geq 1$:
\begin{align*}
&V\ev{D_n} = \{ \ev{s_1,\ldots,s_n} : s_i \in \{0,1\}, \ \forall i \in [n] \}  \\ 
&E\ev{D_{n}} = \{ \ev{ \ev{s_1,\ldots,s_n}, \ev{t_1,\ldots,t_n} } : s_i = t_{i-1} , \ \forall i \in [2,n]\}
\end{align*}
It is known that $D_{n}$ is Eulerian, has $2^{n}$ vertices, $2^{n+1}$ edges and $2$ loops. Eulerian circuits of $D_{n}$ are in one-to-one correspondence with De Bruijn sequences of order $n+1$, on a size-2 alphabet. By \cite[p.~196]{li_zhang}, there are $|\mathfrak{C}\ev{D_{n}}| =  2^{2^{n}-(n+1)} $ Eulerian circuits. For $n\geq 2$, note that $|\mathfrak{C}\ev{D_{n}}| > 1 $. By \Cref{prop:unique_eulerian} and the previous remark, an Eulerian circuit of $D_{n}$ has an interlacing pair $\ev{a,b}$, for any $n\geq 2$. As the maximum degree of $D_n$ is $2$, it also follows by (\ref{eq:3}) that $a\neq b$. In other words, in any cyclic De Bruijn sequence of order $n+1\geq 3$ on a size-$2$ alphabet, we can find a pair $\ev{a,b}$ of distinct strings of length $n$ such that they appear in the order $a-b-a-b$. 
\end{remark}

\section{Eulerian Multigraphs with a Unique Partition into Undirected Cycles}
\label{sec:multi_graph_partition}
In this section, we consider partitions of a connected loopless multigraph into ``undirected cycles". A connected multigraph $Y$ is an \textit{undirected cycle}, if each vertex of $Y$ has degree $2$. The \textit{length} of an undirected cycle $Y$ is the number of edges of $Y$. The smallest undirected cycle is of length $2$, called the \textit{undirected digon}. As in \Cref{def:bridgeless_cactus}, the closure of the collection of finite undirected cycles under the operation $\ast$ is denoted by $\mathcal{S}$. This does not generate any confusion, since they describe different type of objects; directed multigraphs in one case, and undirected multigraphs in the other. Given a multigraph $X$, then we define:
$$
\mathcal{F}\ev{X} := \{ Y: Y \text{ is a subgraph of }X \text{ and } Y \text{ is an undirected cycle}\}
$$
When $X$ is a multigraph, we refer to the elements of $\mathcal{B}\ev{X}$ as ``directed cycles", to distinguish them from the undirected cycles of $X$.
\begin{remark}
Directed cycles of a multigraph $X$ are in bijection with orientations of its undirected cycles: 
$$
\mathcal{B}\ev{X} \leftrightarrow \bigsqcup_{ Y \in \mathcal{F} \ev{X} } \mathcal{O}\ev{Y}
$$
where an undirected cycle $Y$ of length $2$ has a unique orientation, which is a (directed) digon and undirected cycles of length $\geq 3$ have $2$ orientations each. 
\end{remark}
\begin{remark}
\label{rmk:cycles_orientation_restriction}
Let $X$ be a fixed multigraph. For any orientation $O\in \mathcal{O}\ev{X}$:
\begin{enumerate}
\item There is a natural injection,
$$
\begin{tikzcd}
\mathcal{B}\ev{ O } \arrow[r,shift left=2pt,"f"] & \mathcal{F}\ev{ X } 
\end{tikzcd} 
$$
where $f\ev{\beta}$ is the undirected cycle underlying $\beta$.
\item $X\in \mathcal{S}$ if and only if $O\in\mathcal{S}$.
\end{enumerate}
\end{remark}

The \textit{cyclomatic number} or \textit{circuit rank}, $\mathbb{C}\ev{X}$, of a connected multigraph $X$ is the smallest number of edges which must be removed such that no cycle remains (see \cite{wolfram_circuit_rank}). In other words, $\mathbb{C}\ev{X} = \min\{|A|: A\subseteq E\ev{X} \text{ and } \mathcal{F}\ev{E\ev{X} \setminus A} = \emptyset \}$. The cyclomatic number of a connected multigraph can be calculated by the formula, $\mathbb{C}\ev{X}=|E\ev{X}|-|V\ev{X}|+1$ (see \cite[Corollary~4.5, p.~39]{harary_cycle_rank} and \cite[p.~26, Theorem 1.9.57]{diestel}). Clearly, $\mathbb{C}\ev{X} \leq |\mathcal{F}\ev{X}|$ holds. It is well-known (\cite[p.~137]{sedlar}, \cite[p.~153]{volkmann}) that, for a simple graph $X$, equality holds if and only if $X$ is a cactus graph.

For a connected multigraph $X$, the following are equivalent:
\begin{enumerate}
\item[\textit{(1)}] $X$ has a strong orientation (an orientation that is strongly connected).
\item[\textit{(2)}] $X$ is bridgeless. 
\item[\textit{(3)}] Every edge of $X$ is contained in at least one (undirected) cycle. 
\end{enumerate}
where the equivalence \textit{(1)}$\Longleftrightarrow$\textit{(2)} is due to Robbin's Theorem (\cite[p.~281]{robbins}). If $X$ is Eulerian, then it is bridgeless.

We end this section with the characterization of multigraphs with a unique partition into (undirected) cycles. 
\begin{corollary}
\label{cor:unique_decomp_equiv_multi_graph}
Given a connected Eulerian multigraph $X$, then the following are equivalent:
\begin{enumerate}
\item[\textit{(1)}]	No two cycles of $X$ share an edge, i.e., $X$ is a (bridgeless) cactus multigraph.
\item[\textit{(2)}] Every edge of $X$ is contained in exactly one cycle. 
\item[\textit{(3)}] $\mathcal{F}\ev{X} $ is a partition of $X$ into cycles.
\item[\textit{(4)}] $X$ has a unique partition into cycles.
\item[\textit{(5)}] $X \in \mathcal{S}$. 
\item[\textit{(6)}] Cyclomatic number of $X$ is equal to the number of cycles, i.e. $\mathbb{C}\ev{X} = |\mathcal{F}\ev{X}|$. 
\end{enumerate}
\end{corollary}
\begin{proof}
The equivalences \textit{(1)} $\Longleftrightarrow$ \textit{(i)}, for $i=2,3,6$ are clear.

\textit{(3)} $\Longrightarrow$ \textit{(5)}: Let $O\in \mathcal{O}\ev{X}$ be any orientation. By \Cref{rmk:cycles_orientation_restriction} Part 1, the cycles of $O$ are mapped to cycles of $X$ injectively, so $\mathcal{B}\ev{O}$ is a partition of $O$. By \Cref{lem:unique_decomp_equiv_treelike} and \Cref{rmk:cycles_orientation_restriction} Part 2, we obtain $X\in \mathcal{S}$. 

\textit{(5)} $\Longrightarrow$ \textit{(1)}, \textit{(5)} $\Longrightarrow$ \textit{(4)}: By induction on $|\mathcal{F}\ev{X}|$. 

\textit{(4)} $\Longrightarrow$ \textit{(5)}: Let $O\in \mathcal{O}\ev{X}$ be any orientation. The injection in \Cref{rmk:cycles_orientation_restriction} Part 1 induces an injection of partitions:
\vspace{-2mm}
$$
\begin{tikzcd}
\{ \mathcal{A} \subseteq \mathcal{B}\ev{ O } : \text{ $\mathcal{A}$ is a partition of $O$} \}  \arrow[r,shift left=2pt,"\widehat{f}"] & \{ \mathcal{A} \subseteq \mathcal{F}\ev{ X } : \text{ $\mathcal{A}$ is a partition of $X$} \} 
\end{tikzcd}
$$

In particular, $O$ has a unique partition into cycles. Again, the statement follows from \Cref{lem:unique_decomp_equiv_treelike} and \Cref{rmk:cycles_orientation_restriction} Part 2.
\end{proof}


\bibliography{bibliography.bib}

\begin{thebibliography}{25}
\providecommand{\natexlab}[1]{#1}
\providecommand{\url}[1]{\texttt{#1}}
\expandafter\ifx\csname urlstyle\endcsname\relax
  \providecommand{\doi}[1]{doi: #1}\else
  \providecommand{\doi}{doi: \begingroup \urlstyle{rm}\Url}\fi

\bibitem[Acosta and Tomescu(2024)]{acosta}
N.~O. Acosta and A.~I. Tomescu.
\newblock Simplicity in eulerian circuits: Uniqueness and safety.
\newblock \emph{Information Processing Letters}, 183:\penalty0 106421, 2024.
\newblock ISSN 0020-0190.

\bibitem[Angelini et~al.(2010)Angelini, Frati, and Grilli]{angelini}
P.~Angelini, F.~Frati, and L.~Grilli.
\newblock An algorithm to construct greedy drawings of triangulations.
\newblock \emph{Journal of Graph Algorithms and Applications}, 14\penalty0
  (1):\penalty0 19--51, 2010.

\bibitem[Arratia et~al.(2004)Arratia, Bollobás, and Sorkin]{bollobas}
R.~Arratia, B.~Bollobás, and G.~B. Sorkin.
\newblock The interlace polynomial of a graph.
\newblock \emph{Journal of Combinatorial Theory, Series B}, 92\penalty0
  (2):\penalty0 199--233, 2004.
\newblock ISSN 0095-8956.
\newblock Special Issue Dedicated to Professor W.T. Tutte.

\bibitem[Bahrani and Lumbroso(2018)]{bahrani}
M.~Bahrani and J.~Lumbroso.
\newblock Split-decomposition trees with prime nodes: Enumeration and random
  generation of cactus graphs.
\newblock In M.~Nebel and S.~Wagner, editors, \emph{2018 Proceedings of the
  15th Workshop on Analytic Algorithmics and Combinatorics, ANALCO 2018}, pages
  143--157, United States, 2018. Society for Industrial and Applied Mathematics
  Publications.

\bibitem[Bondy and Murty(2008)]{bondy}
J.~A. Bondy and U.~S.~R. Murty.
\newblock \emph{Graph Theory}.
\newblock Springer Publishing Company, 1st edition, 2008.

\bibitem[Cooper and Okur(2025)]{co3}
J.~Cooper and U.~Okur.
\newblock Partitions of an eulerian digraph into circuits, 2025.

\bibitem[Diestel(2017)]{diestel}
R.~Diestel.
\newblock \emph{Graph Theory}.
\newblock Springer Publishing Company, Incorporated, 5th edition, 2017.

\bibitem[{G}raham et~al.(1995){G}raham, {G}rötschel, and (eds.)]{helly}
R.~{G}raham, M.~{G}rötschel, and L.~L. (eds.).
\newblock \emph{Handbook of Combinatorics}.
\newblock Elsevier, 1995.

\bibitem[Gross et~al.(2019)Gross, Yellen, and Anderson]{gross}
J.~L. Gross, J.~Yellen, and M.~Anderson.
\newblock \emph{Graph Theory and Its Applications}.
\newblock Textbooks in Mathematics. CRC Press, Boca Raton, FL, 3rd edition,
  2019.
\newblock First issued in paperback.

\bibitem[Harary(1969)]{harary_cycle_rank}
F.~Harary.
\newblock \emph{Graph Theory}.
\newblock Addison-Wesley series in mathematics. Addison-Wesley Publishing
  Company, 1969.
\newblock ISBN 9788185015552.

\bibitem[Harary and Uhlenbeck(1953)]{husimi}
F.~Harary and G.~E. Uhlenbeck.
\newblock On the number of {H}usimi trees, {I}.
\newblock \emph{Proceedings of the National Academy of Sciences of the United
  States of America}, 39\penalty0 (4):\penalty0 315--322, 1953.

\bibitem[Leighton and Moitra(2010)]{leighton}
T.~Leighton and A.~Moitra.
\newblock Some results on greedy embeddings in metric spaces.
\newblock \emph{Discrete \& Computational Geometry}, 44\penalty0 (3):\penalty0
  686--705, 10 2010.

\bibitem[Li and Zhang(1991)]{li_zhang}
X.~Li and F.~Zhang.
\newblock On the numbers of spanning trees and eulerian tours in generalized de
  bruijn graphs.
\newblock \emph{Discrete Mathematics}, 94\penalty0 (3):\penalty0 189--197,
  1991.
\newblock ISSN 0012-365X.

\bibitem[McKee and McMorris(1999)]{mckee}
T.~A. McKee and F.~R. McMorris.
\newblock \emph{Topics in Intersection Graph Theory}.
\newblock Discrete Mathematics and Applications. Society for Industrial and
  Applied Mathematics, 1999.
\newblock ISBN 9780898719802.

\bibitem[Pevzner(1989)]{pevzner}
P.~A. Pevzner.
\newblock l-tuple dna sequencing: computer analysis.
\newblock \emph{{J}ournal of {B}iomolecular structure and dynamics}, 7\penalty0
  (1):\penalty0 63--73, 1989.

\bibitem[Prutkin(2016)]{prutkin}
R.~Prutkin.
\newblock A note on the area requirement of euclidean greedy embeddings of
  {C}hristmas cactus graphs, 2016.

\bibitem[Robbins(1939)]{robbins}
H.~E. Robbins.
\newblock A theorem on graphs, with an application to a problem of traffic
  control.
\newblock \emph{The American Mathematical Monthly}, 46\penalty0 (5):\penalty0
  281--283, 1939.
\newblock ISSN 00029890, 19300972.

\bibitem[Schrijver(2003)]{schrijver}
A.~Schrijver.
\newblock \emph{Combinatorial Optimization: Polyhedra and Efficiency}.
\newblock Number v. 1 in Algorithms and Combinatorics. Springer, 2003.
\newblock ISBN 9783540443896.

\bibitem[Sedlar and Škrekovski(2022)]{sedlar}
J.~Sedlar and R.~Škrekovski.
\newblock Vertex and edge metric dimensions of cacti.
\newblock \emph{Discrete Applied Mathematics}, 320:\penalty0 126--139, 2022.
\newblock ISSN 0166-218X.

\bibitem[van Aardenne-Ehrenfest and de~Bruijn(1951)]{best}
T.~van Aardenne-Ehrenfest and N.~de~Bruijn.
\newblock Circuits and trees in oriented linear graphs.
\newblock \emph{Simon Stevin: Wis-en Natuurkundig Tijdschrift}, 28:\penalty0
  203--217, 1951.
\newblock ISSN 0037-5454.

\bibitem[Veblen(1912)]{veblen}
O.~Veblen.
\newblock An application of modular equations in analysis situs.
\newblock \emph{Annals of Mathematics}, 14\penalty0 (1/4):\penalty0 86--94,
  1912.

\bibitem[Volkmann(1996)]{volkmann}
L.~Volkmann.
\newblock Estimations for the number of cycles in a graph.
\newblock \emph{Periodica Mathematica Hungarica}, 33\penalty0 (2):\penalty0
  153--161, 1996.
\newblock ISSN 1588-2829.

\bibitem[Waterman(1995)]{waterman}
M.~S. Waterman.
\newblock \emph{Introduction to Computational Biology: Maps, Sequences and
  Genomes}.
\newblock Chapman \& Hall/CRC Interdisciplinary Statistics. Taylor \& Francis,
  1995.
\newblock ISBN 9780412993916.

\bibitem[Weisstein(2023{\natexlab{a}})]{weisstein}
E.~W. Weisstein.
\newblock Cactus graph.
\newblock 2023. From MathWorld--A Wolfram Web Resource. URL:
  \url{https://mathworld.wolfram.com/CactusGraph.html}. Accessed: 2025-1-14,
  2023{\natexlab{a}}.

\bibitem[Weisstein(2023{\natexlab{b}})]{wolfram_circuit_rank}
E.~W. Weisstein.
\newblock Circuit rank.
\newblock 2023. From MathWorld--A Wolfram Web Resource. URL:
  \url{https://mathworld.wolfram.com/CircuitRank.html}. Accessed: 2025-1-14,
  2023{\natexlab{b}}.

\end{thebibliography}

\end{document}